\documentclass[11pt]{article}

\usepackage{subfigure}
\usepackage{tikz}

\usetikzlibrary{graphs,graphs.standard,calc}
\usepackage[utf8]{inputenc}

\usepackage[utf8]{inputenc}
\usepackage[english]{babel}
\usepackage{amsmath}
\usepackage{times}
\usepackage{graphicx}
\usepackage{float}
\usepackage[margin=1in]{geometry}
\usepackage{blindtext}
\usepackage{amssymb}
\usepackage{tikz}
\usepackage{amsthm}
\usepackage{mathtools}
\usepackage{textcmds}

\usepackage{comment}
\usepackage[mathlines]{lineno}
\newtheorem{theorem}{Theorem}[section]

\newtheorem{lemma}[theorem]{Lemma}
\newtheorem{conjecture}[theorem]{Conjecture}
\theoremstyle{remark}

\newtheorem{claim}{Claim}[theorem]
\theoremstyle{definition}

\title{A spectral Erd\H{o}s-S\'os theorem}
\date{\today}
\author{Sebastian Cioab\u{a}\thanks{Department of Mathematical Sciences, University of Delaware, \texttt{cioaba@udel.edu}. Research partially supported by National Science Foundation grant CIF-1815922.} \and Dheer Noal Desai\thanks{Department of Mathematical Sciences, University of Delaware, \texttt{dheernsd@udel.edu}} \and Michael Tait\thanks{Department of Mathematics \& Statistics, Villanova University, \texttt{michael.tait@villanova.edu}. Research partially supported by National Science Foundation grant DMS-2011553.}}

\begin{document}

\maketitle
\begin{abstract}
   The famous Erd\H{o}s-S\'os conjecture states that every graph of average degree more than $t-1$ must contain every tree on $t+1$ vertices. In this paper, we study a spectral version of this conjecture. For $n>k$, let $S_{n,k}$ be the join of a clique on $k$ vertices with an independent set of $n-k$ vertices and denote by $S_{n,k}^+$ the graph obtained from $S_{n,k}$ by adding one edge. We show that for fixed $k\geq 2$ and sufficiently large $n$, if a graph on $n$ vertices has adjacency spectral radius at least as large as $S_{n,k}$ and is not isomorphic to $S_{n,k}$, then it contains all trees on $2k+2$ vertices. Similarly, if a sufficiently large graph has spectral radius at least as large as $S_{n,k}^+$, then it either contains all trees on $2k+3$ vertices or is isomorphic to $S_{n,k}^+$. This answers a two-part conjecture of Nikiforov affirmatively. 
\end{abstract}
\section{Introduction}
For a fixed graph $F$, the {\em Tur\'an number} for $F$ is denoted by $\mathrm{ex}(n, F)$ and is the maximum number of edges an $n$ vertex $F$-free graph may have. The celebrated Erd\H{o}s-Stone-Simonovits theorem \cite{ES66, ES46} determines an asymptotic formula for $\mathrm{ex}(n, F)$ for any $F$ with chromatic number at least $3$. On the other hand, determining even the order of magnitude of the Tur\'an number for most bipartite graphs is a notoriously difficult problem. 

The particular class of bipartite Tur\'an problems of interest to us in this paper is the class of trees. Given a tree $T$ with $t+1$ vertices, considering disjoint copies of $K_t$ shows that $\mathrm{ex}(n, T) \geq \left \lfloor \frac{n}{t} \right \rfloor \binom{t}{2}$. On the other hand, a standard graph theory exercise shows that $\mathrm{ex}(n, T) \leq n(t-1)$. Note the multiplicative factor of $2$ between these bounds. The Erd\H{o}s-S\'os conjecture predicts that the lower bound is correct for all trees.

\begin{conjecture}[Erd\H{o}s and S\'os 1963 (see \cite{erdos1965})]
Any graph of average degree larger than $t-1$ contains all trees on $t+1$ vertices.
\end{conjecture}

There are many partial results towards this conjecture. For example, the conjecture is true if one additionally forbids $C_4$ (see \cite{SW} which improved the result from \cite{brandt} that the conjecture holds in host graphs of girth $5$), in the case that $t$ is large compared to $n$ \cite{GZ, STY, Tiner, wozniak}, when the tree has diameter $4$ \cite{M}. An asymptotic solution for sufficiently many vertices was announced by Ajtai-Koml\'os-Simonovits-Szemer\'edi (c.f. \cite[pp. 224-227]{degeneratesurvey}). Many other partial results towards the conjecture have appeared and we will not try to list them all.

In this paper we consider a spectral version of this conjecture. Given a graph $G$, let $\lambda=\lambda(G)$ denote the spectral radius of its adjacency matrix. Maximizing $\lambda$ over a fixed family of graphs is known as a {\em Brualdi-Solheid problem} \cite{BS}. Many results on Brualdi-Solheid type problems have been obtained over the years, for example \cite{BZ, BLL, EZ, FN, nosal1970eigenvalues, S, SAH}. A Brualdi-Solheid problem inspired by Tur\'an numbers is to ask to maximize $\lambda$ over the family of $n$-vertex $F$-free graphs. Indeed, this was proposed to be systematically investigated by Nikiforov in \cite{Nikiforovpaths}.

We define $\mathrm{spex}(n, F)$ to be the maximum adjacency spectral radius that an $n$-vertex graph not containing a copy of $F$ may have. Determining $\mathrm{spex}(n, F)$ for various graphs $F$ has been of significant interest since being introduced, for example \cite{evencycles, cioabua2022spectral, FG, SpectralIntersectingCliques, Yongtao21, LLT, Nikiforov07, NO,  SSbook, Wilf86, YWZ, ZW, ZWF}. One motivation for studying this parameter is that because the spectral radius of a graph is an upper bound for its average degree, any upper bound on $\mathrm{spex}(n, F)$ also gives an upper bound on $\mathrm{ex}(n,F)$ via the inequality
\[
\mathrm{ex}(n, F) \leq \frac{n}{2}\mathrm{spex}(n, F).
\]

Several classical extremal graph theory results have spectral strengthenings through this inequality, for example Tur\'an's theorem is implied by Nikiforov's result \cite{Nikiforov07} on $\mathrm{spex}(n, K_{r+1})$. In this paper we study $\mathrm{spex}(n, T)$ when $T$ is a tree. Note that disjoint copies of $K_t$ show that $\mathrm{spex}(n, T) \geq t-1$ if $T$ is a tree with $t+1$ vertices. However, we can do much better than this. For $n>k$, let $S_{n,k}$ be the join of a clique on $k$ vertices with an independent set of $n-k$ vertices and denote by $S_{n,k}^+$ the graph obtained from $S_{n,k}$ by adding one edge. If a tree has at least $k+1$ vertices in each side of its bipartition then it is not a subgraph of $S_{n,k}$. Therefore, a path on $2k+2$ vertices is not a subgraph of $S_{n,k}$ and a path on $2k+3$ vertices is not a subgraph of $S_{n,k}^+$. Nikiforov \cite{Nikiforovpaths} showed that in fact $S_{n,k}$ and $S_{n,k}^+$ have the largest spectral radii over all $n$-vertex graphs forbidding paths on $2k+2$ and $2k+3$ vertices respectively (see Lemma \ref{bounds on lambda} for the value of $\lambda(G_{n,k})$). In the same paper, it was conjectured that these graphs are spectral extremal for the set of all trees on $2k+2$ or $2k+3$ vertices, yielding the following spectral version of the Erd\H{o}s-S\'os conjecture.

\begin{conjecture}[Nikiforov 2010]
\label{conjecture Nikiforov paths}
Let $k \geq 2$ and $G$ be a graph of sufficiently large $n$.  \\
(a) if $\lambda_1(G) \geq \lambda_1(S_{n,k})$ then $G$ contains all trees of order $2k + 2$ unless $G = S_{n,k}$; \\
(b) if $\lambda_1(G) \geq \lambda_1(S_{n,k}^+)$ then $G$ contains all trees of order $2k + 3$ unless $G = S_{n,k}^+$.
\end{conjecture}

Partial results towards this conjecture were given in \cite{hou, LBW2, LBW} for trees of diameter at most $4$ and other special cases. In this paper we prove Conjecture \ref{conjecture Nikiforov paths}. The proof builds on our work \cite{evencycles} in which we determined $\mathrm{spex}(n, C_{2\ell})$ for sufficiently large $n$.

\section{Organization and Notation}

For $k \geq 2$, let $\mathcal{T}_k$ denote the set of all trees on $2k+2$ vertices and $\mathcal{T}_k'$ denote the set of all trees on $2k+3$ vertices. We use $\mathcal{G}_{n,k}$ to denote the set of graphs on $n$ vertices not containing at least one tree $T \in \mathcal{T}_k$ on $2k+2$ vertices. Let $G_{n,k}$ be a graph with maximum spectral radius among all graphs in $\mathcal{G}_{n,k}$. Similarly, we use $\mathcal{G}_{n,k}'$ to denote the set of graphs on $n$ vertices not containing at least one tree $T \in \mathcal{T}_k'$ on $2k+3$ vertices, and $G_{n,k}'$ to denote a graph with maximum spectral radius among all graphs in $\mathcal{G}_{n,k}'$. 

Let $\mathrm{x}$ be the scaled Perron vector whose maximum entry is equal to $1$ and let $z$ be a vertex having Perron entry $\mathrm{x}_z = 1$. We define positive constants $\eta, \epsilon$ and $\alpha$ satisfying 
\begin{align}\label{choice of constants}
\begin{split}
\eta & <\left\{\frac{1}{10k}, \frac{1}{(2k+2)(16k^2)} \cdot \left( 16k^2-1 - \frac{4(16k^3-1)}{5k}\right)\right\}\\
\epsilon &< \min\left\{\eta, \frac{\eta}{2}, \frac{1}{8k^3}, \frac{\eta}{32k^3+2}\right\}\\
\alpha &< \min \left\{ \eta, \frac{\epsilon^2}{22k} \right\}.
\end{split}
\end{align}

We leave the redundant conditions in these inequalities so that when we reference this choice of constants it is clear what we are using. Let $L$ be the set of vertices in $V(G_{n,k})$ having ``large" weights and $S = V \setminus L$ be the vertices of ``small" weights, \[L := \{v \in V | \mathrm{x}_v \geq \alpha\}, S := \{v \in V | \mathrm{x}_v < \alpha\}.\]
 Also, let $M \supset L$ be the set of vertices of ``not too small" a weight: \[M := \{v \in V | \mathrm{x}_v\geq \alpha/3\}.\] Denote by $L' \subset L$ the following subset of vertices of ``very large'' weights: \[L' := \{v \in L | \mathrm{x}_v \geq \eta\}.\]

Finally, let $N_i(u)$ be the vertices at distance $i$ from some vertex $u \in V$, and $L_i(u), S_i(u)$ and $M_i(u)$ be the sets of vertices in $L, S$ and $M$ respectively, at distance $i$ from $u$. If the vertex is unambiguous from context, we will use $L_i$, $S_i,$ and $M_i$ instead. 

In Section \ref{section preliminary lemmas} we give some background graph theory lemmas that we will need and prove preliminary lemmas on the eigenvalues and eigenvectors of our extremal graphs. In Section \ref{section structural results for extremal graphs} we progressively refine the structure of our extremal graphs. Using these results, we prove Conjecture \ref{conjecture Nikiforov paths} in Section \ref{section main proof}.

\section{Lemmas from spectral and extremal graph theory}
\label{section preliminary lemmas}

\begin{lemma}
\label{lower bound for spectral radius of nonnegative matrices}
For a non-negative symmetric matrix $B$, a non-negative non-zero vector $\mathbf{\mathrm{y}}$ and a positive constant $c$, if $B \mathbf{\mathrm{y}} \geq c \mathbf{\mathrm{y}}$ entrywise, then $\lambda(B) \geq c$. 
\end{lemma}

\begin{proof}
Assume that $B \mathbf{\mathrm{y}} \geq c \mathbf{\mathrm{y}}$ entrywise, with the same assumptions for $B, \mathbf{\mathrm{y}}$ and $c$ as in the statement of the theorem. Then $\mathbf{\mathrm{y}}^T B \mathbf{\mathrm{y}} \geq \mathbf{\mathrm{y}}^T c \mathbf{\mathrm{y}}$ and $\lambda(B) \geq \dfrac{\mathbf{\mathrm{y}}^T B \mathbf{\mathrm{y}}}{\mathbf{\mathrm{y}}^T  \mathbf{\mathrm{y}}} \geq c$.
\end{proof}

\begin{lemma}
\label{bounds on turan number for trees}
For any tree $T_t$ on $t$ vertices,
\begin{equation}
    \frac{1}{2}(t-2)n \leq \mathrm{ex}(n, T_t) \leq (t-2)n.
\end{equation}
\end{lemma}

\begin{lemma}
\label{trees growing in bipartite graph}
For $m \ge 2$, the complete bipartite graph $K_{\left\lfloor\frac{m}{2}\right\rfloor, m-1}$ contains all trees on $m$ vertices.
\end{lemma}
\begin{proof}
Let $T_m$ be a tree with $m$ vertices. We show that $T_m$ is contained in $K_{\left\lfloor\frac{m}{2}\right\rfloor, m-1}$. The tree $T_m$ is bipartite with the smaller part having at least one vertex and at most $\left\lfloor\frac{m}{2}\right\rfloor$ vertices, while the larger part has at least $\left\lceil\frac{m}{2}\right\rceil$ vertices and at most $m-1$ vertices. If we use a one-to-one map to send the vertices of $T_m$ into $K_{\left\lfloor\frac{m}{2}\right\rfloor, m-1}$ so that the smaller part of $T_m$ maps into the part on $\left\lfloor\frac{m}{2}\right\rfloor$ vertices and the larger part of $T_m$ maps into the other part of the complete bipartite graph, then we can see that $T_m \subset K_{\left\lfloor\frac{m}{2}\right\rfloor, m-1}$
\end{proof}

\begin{lemma}\label{bounds on lambda}
$\sqrt{kn} \leq \dfrac{k-1 + \sqrt{(k-1)^2 + 4k(n-k)}}{2} \leq \lambda(G_{n,k}) \leq \lambda(G_{n,k}') < \sqrt{(4k+2)n} \leq \sqrt{5kn}$.
\end{lemma}
\begin{proof}
We get the lower bound from the spectral radius of $S_{n,k}$. The upper bound is obtained by using the second-degree eigenvalue-eigenvector equation and $e(G_{n,k}') \leq (2k+1)n$ (from  Lemma \ref{bounds on turan number for trees}):
\begin{equation}
    \begin{split}
        \lambda^2 &\leq \max_{v \in V} \left\{\sum_{u \sim v} A^2_{v, u}\right\} = \max_{v \in V}\left\{\sum_{u \sim v} d(u)\right\}\\ 
        & = \max_{v \in V}\left\{d(v) + 2e(N_1(v)) + e(N_1(v), N_2(v))\right\}\\
        & \leq 2e(G_{n,k}') \leq (4k+2)n \leq 5kn,
    \end{split}
\end{equation}
which finishes our proof.
\end{proof}

\begin{lemma}\label{lemma:LM}
$|L| \leq \dfrac{5\sqrt{kn}}{\alpha}$ and $|M| \leq \dfrac{15\sqrt{kn}}{\alpha}$.
\end{lemma}
\begin{proof}
For any $v \in V$, we have that 
\begin{equation}
\label{eigenvalue-eigenvector equation}
    \lambda \mathrm{x}_v = \sum_{u \sim v} \mathrm{x}_u \leq d(v).
\end{equation}
 Combining this with the lower bound in Lemma \ref{bounds on lambda} gives $\sqrt{kn} \mathrm{x}_v \leq \lambda \mathrm{x}_v \leq d(v)$. Summing over all vertices $v \in L$ gives
\begin{equation}
 |L|\sqrt{kn}\alpha \leq \sum_{v\in L} d(v) \leq \sum_{v \in V} d(v) \leq 2 e(G_{n, k}') \leq 2 \max_{T \in \mathcal{T}_k'}\{\mathrm{ex}(n, T)\} \leq (4k+2)n \le 5kn.
\end{equation}
Thus, $|L| \leq \dfrac{5\sqrt{kn}}{\alpha}$. A similar argument shows that $|M| \leq \dfrac{15\sqrt{kn}}{\alpha}$.\end{proof}

\begin{lemma}
The graphs $G_{n, k}$ and $G_{n, k}'$ are connected and for any vertex $v \in V(G'_{n,k})$ (or $V(G_{n,k})$), we have that 
\begin{equation*}
\mathrm{x}_v \geq \dfrac{1}{\lambda(G_{n,k})} > \dfrac{1}{\sqrt{5kn}}.
\end{equation*}
\end{lemma}
\begin{proof}
For brevity, we prove this for $G_{n,k}'$, the proof for $G_{n, k}$ being identical. If $G_{n,k}'$ has $r\geq 2$ components: $C_1, \ldots, C_r$, then let $C_1$ be a component containing $z$ such that $\lambda(G_{n,k}') = \lambda(C_1)$. 
By applying \eqref{eigenvalue-eigenvector equation} for the vertex $z$, we have that $d(z) \ge \lambda(G_{n, k}') > \sqrt{kn}$.
Say $u$ is some vertex in another component. We modify the graph $G_{n,k}'$ by deleting all the edges adjacent to $u$ and adding the edge $uz$, to obtain the graph $\hat{G}_{n,k}$. As $\lambda(C_1)=\lambda(G_{n,k}')$, this modification strictly increases the spectral radius. Moreover, we claim that $\hat{G}_{n,k}$ does not contain any new trees in $\mathcal{T}'_k$ which were not contained in $G_{n,k}'$. Assume to the contrary that a new tree $T \in \mathcal{T}_k'$ is created due to the modification and that there are no isomorphic copies of $T$ in $G_{n,k}'$. Then $u$ is a leaf of $T$ and is adjacent only to the vertex $z$. 
Now $T \in \mathcal{T}'_k$ is a graph on $2k+3$ vertices and $d(z) \ge \sqrt{kn}$. So, $z$ would have already been adjacent in $G'_{n,k}$ to at least $\sqrt{kn} - (2k+1)$ vertices not in the vertex set of $T$, thus an isomorphic copy of $T$ is already present in $G'_{n,k}$, a contradiction. 

The statement follows from \eqref{eigenvalue-eigenvector equation} for any vertex $v$ adjacent to $z$. Now consider some vertex $v \in V$  that is not adjacent to $z$ such that $\mathrm{x}_v < \dfrac{1}{\lambda(G_{n, k}')}$. As above, we have $d(z) \ge \sqrt{kn}$.  Modifying the graph $G_{n,k}'$ in the same way it was done above,  by deleting the edges adjacent to $v$ in $G_{n, k}'$ and adding the edge $vz$, strictly increases the spectral radius by considering the Rayleigh quotient. For the same reason as before, we also do not create any new $T \in \mathcal{T}_k'$ which had no isomorphic copies in $G_{n,k}'$ already. 
\end{proof}
\section{Structural results for extremal graphs}
\label{section structural results for extremal graphs}

All the lemmas in this section apply to both $G_{n,k}$ and $G_{n,k}'$. For brevity, we write all the proofs for $G_{n,k}'$, but the same results for $G_{n,k}$ follow by replacing $G_{n,k}'$ with $G_{n,k}$ in the proof. First we show that the degree of any vertex in $L$ is linear and hence $L$ has constant size.

\begin{lemma}
\label{degrees of vertices in L}
Any vertex $v \in L$ has degree $d(v) \ge \dfrac{\alpha}{10(4k+3)}n$. Moreover, $|L| \leq \dfrac{500}{\alpha}k^2$.
\end{lemma}

\begin{proof}
Assume to the contrary that there exists a vertex $v \in L$ of degree $d(v)<\dfrac{\alpha}{10(4k+3)}n$. Let $\mathrm{x}_v = c \ge \alpha$. The second degree eigenvector-eigenvalue equation with respect to the vertex $v$ gives that
\begin{equation}
\label{lower bound on degrees in L}
    \begin{split}
        knc \leq \lambda^2 c = \lambda^2 \mathrm{x}_v = \sum_{\substack{u \sim v \\ w \sim u}} \mathrm{x}_w 
        & \leq d(v)c + 2e(N_1(v)) + \sum_{u \sim v}\sum_{\substack{w \sim u, \\ w \in N_2(v)}}\mathrm{x}_w\\
        & \leq (4k+3)d(v) + \sum_{u \sim v}\sum_{\substack{w \sim u, \\ w \in M _2(v)}}\mathrm{x}_w + \sum_{u \sim v}\sum_{\substack{w \sim u, \\ w \in N_2 \setminus M_2(v)}}\mathrm{x}_w,
    \end{split}
\end{equation}
where the last inequality follows from Lemma \ref{bounds on turan number for trees}.
Since $d(v) < \dfrac{\alpha}{10(4k+3)}n$ by our assumption and $c \ge \alpha$, we have
\begin{equation}
    (k - 0.1)nc <  \sum_{u \sim v}\sum_{\substack{w \sim u, \\ w \in M _2(v)}}\mathrm{x}_w + \sum_{u \sim v}\sum_{\substack{w \sim u, \\ w \in N_2 \setminus M_2(v)}}\mathrm{x}_w.
\end{equation}
From $d(v)<\frac{\alpha}{10(4k+3)}$ and Lemma \ref{lemma:LM}, we get that
 \[\sum_{u \sim v}\sum_{\substack{w \sim u, \\ w \in M_2(v)}}\mathrm{x}_w \leq e(N_1(v), M_2(v)) \leq (2k+1)(d(v) + |M|) < (2k+1)\left(\dfrac{\alpha}{10(4k+3)}n + \dfrac{15\sqrt{kn}}{\alpha}\right).\]
For $n$ sufficiently large, we have that 
\[
(2k+1)\left(\dfrac{\alpha}{10(4k+3)}n + \dfrac{15\sqrt{kn}}{\alpha}\right) \leq .9n\alpha \leq .9nc,
\]
and consequently using Lemma \ref{bounds on turan number for trees},
\[(k-1)nc < \sum_{u \sim v}\sum_{\substack{w \sim u, \\ w \in N_2 \setminus M_2(v)}}\mathrm{x}_w \leq e(N_1(v), N_2 \setminus M_2(v))\dfrac{\alpha}{3} \leq (2k-1)n \dfrac{\alpha}{3}.\]
This is a contradiction because $k \geq 2$ and $c \ge \alpha$. Hence, $d(v) \geq \dfrac{\alpha}{10(4k+3)}n$ for all $v \in L$. 
Thus, \[(2k+1)n \ge e(G'_{n, k}) \ge \dfrac{1}{2}\sum_{v\in L}d(v) \ge |L|\dfrac{\alpha}{20(4k+3)}n,\]
so $|L| \leq \dfrac{20(2k+1)(4k+3)}{\alpha} \le \dfrac{500}{\alpha}k^2$.
\end{proof}

We will now refine our arguments in the proof above to improve degree estimates for the vertices in $L'$.

\begin{lemma}
\label{mindegrees for vertices in L'}
If $v$ is a vertex in $L'$ with $\mathrm{x}_v = c$, then $d(v) \ge cn -\epsilon n$. 
\end{lemma}
\begin{proof}
Our proof is again by contradiction. The second degree eigenvalue-eigenvector with respect to the vertex $v$ gives
\begin{equation}
\begin{split}
    knc 
    & \le \lambda^2 c = \sum_{u \sim v} \sum_{w \sim u} \mathrm{x}_w = d(v) c + \sum_{u \sim v} \sum_{\substack{w \sim u,\\ w \neq v}} \mathrm{x}_w\\
    & \le d(v) c + \sum_{u \in S_1} \sum_{\substack{w \sim u \\ w \in L_1 \cup L_2}} \mathrm{x}_w + 2e(S_1)\alpha + 2e(L) + e(L_1, S_1)\alpha + e(N_1, S_2)\alpha.
\end{split}
\end{equation}
The observation that any subgraph of $G_{n, k}'$ has at most $(2k+1)n$ edges implies that
\begin{align*}
    2e(S_1) &\le (4k+2)n,\\
    e(L_1, S_1) &\le (2k+1)n,\\
    e(N_1, S_2) &\le (2k+1)n.
\end{align*}
Combining these inequalities with Lemma \ref{degrees of vertices in L}, we deduce that
\[
    2e(S_1)\alpha + 2e(L) + e(L_1, S_1)\alpha + e(N_1, S_2)\alpha \le (4k+2)n\alpha + 2 \binom{|L|}{2} + (2k+1)n\alpha + (2k+1)n\alpha \le 11kn\alpha,
\]
for $n$ sufficiently large. Hence,  
\begin{equation}
\label{upperbound on lambdasq 2}
    knc \le d(v) c + \sum_{u \in S_1} \sum_{\substack{w \sim u \\ w \in L_1 \cup L_2}} \mathrm{x}_w + 11kn\alpha\le d(v) c + e(S_1, L_1 \cup L_2) + 11kn\alpha \le d(v) c + e(S_1, L_1 \cup L_2) + \frac{\epsilon^2 n}{2},
\end{equation}
where the last inequality is by \eqref{choice of constants}. 
Thus, using $d(v) < cn - \epsilon n$, we get that
\begin{equation}
    \left(k - c + \epsilon\right)nc \le \left(kn - d(v)\right)c \le e(S_1, L_1 \cup L_2) + \frac{\epsilon^2 n}{2}.
\end{equation}

Since $v \in L'$, by \eqref{choice of constants} we have $c \geq \eta \geq \epsilon$. Using $\epsilon \le c \le 1$, we obtain that
\begin{equation}
\label{degree bound using edges in some bipartite graph involving vertices of L}
    e(S_1, L_1 \cup L_2) \geq (k-c)nc + \epsilon nc - \frac{\epsilon^2 n}{2} \geq (k-1)nc + \frac{\epsilon^2 n}{2}.
\end{equation}

We show that $G'_{n,k}$ contains a $K_{k+1, 2k+2}$ which is a contradiction with Lemma \ref{trees growing in bipartite graph}. We first prove the following claim.

\begin{claim}
\label{claim in many edges on (L, MUT)}
If $\delta:= \frac{\epsilon \alpha}{500 k^2}$, then there are at least $\delta n$ vertices inside $S_1$ with degree at least $k$ in $G_{n,k}'[S_1, L_1 \cup L_2]$.
\end{claim}

\begin{proof}
Assume to the contrary that at most $\delta n$ vertices in $S_1$ have degree at least $k$ in $G_{n,k}'[S_1, L_1 \cup L_2]$. Then $e(S_1, L_1 \cup L_2) < (k-1)|S_1| + |L|\delta n \leq (k-1)(c-\epsilon)n +\epsilon n$, because $|S_1| \leq d(z)$ and by Lemma \ref{degrees of vertices in L}. Combining this with \eqref{degree bound using edges in some bipartite graph involving vertices of L} gives $ (k-1)nc- (k-2)n \epsilon > e(S_1, L_1\cup L_2) \geq (k-1)nc + \frac{\epsilon^2 n}{2}$, a contradiction.
\end{proof}
Let $D$ be the set of vertices of $S_1$ that have degree at least $k$ in $G_{n,k}'[S_1, L_1 \cup L_2]$. Thus $|D| \ge \delta n$. Since there are only $\binom{|L|}{k} \le \binom{500 k^2/\alpha}{k}$ options for any vertex in $D$ to choose a set of $k$ neighbors from, it implies that there exists some set of $k$ vertices in $L_1 \cup L_2$ with at least $\delta n/ \binom{|L|}{k} \ge  \frac{\epsilon \alpha n}{500 k^2}/ \binom{500 k^2/\alpha}{k}$ common neighbors in $D$. For $n$ large enough, this quantity is at least $2k+2$, thus $K_{k+1, 2k+2} \subset G_{n,k}'[S_1, L_1 \cup L_2 \cup \{v\}]$, and we get our desired contradiction. 
\end{proof}

Lemma \ref{mindegrees for vertices in L'} implies that $d(z) \geq (1-\epsilon)n$ and that for any $v\in L'$ we have $d(v) \geq (\eta - \epsilon)n$. By \eqref{choice of constants}, the neighborhoods of $z$ and $v$ intersect and hence $L' \subset \{z\} \cup N_1(z) \cup N_2(z)$. We need the following inequalities for the number of edges from the neighborhoods of vertices with Perron weight close to the maximum.

\begin{lemma}
\label{neighborhood of x, L}
For $z$ the vertex with $\mathbf{\mathrm{x}}_z =1$, we have $(1 - \epsilon)kn \leq e(S_1, \{z\} \cup L_1\cup L_2) \leq (k+\epsilon)n$.   
\end{lemma}
\begin{proof}

To obtain the lower bound, we use \eqref{upperbound on lambdasq 2}. Given $ \mathrm{x}_z= 1$, we have 
\[
kn(1) \le d(z) + e(S_1, L_1 \cup L_2) + \frac{\epsilon^2 n}{2}.
\]
Since $e(S_1, \{z\} \cup L_1\cup L_2)= e(S_1, L_1\cup L_2) + d(z)$, the lower bound follows as $\frac{\epsilon^2 n}{2} < k\epsilon n$.

To obtain the upper bound, we assume toward a contradiction that for the vertex $z$, we have $e(S_1, \{z\} \cup L_1\cup L_2)> (k + \epsilon)n$. We will obtain a contradiction via Lemma \ref{trees growing in bipartite graph} by showing that $K_{k+1, 2k+2} \subset G_{n,k}'$. For this we prove the following claim.

\begin{claim}
\label{claim in many edges on (L, MUT) modified}
Let $\delta:= \frac{\epsilon \alpha}{500 k^2}$. With respect to the vertex $z$ there are at least $\delta n$ vertices inside $S_1$ with degree $k$ or more in $G_{n,k}'[S_1, L_1 \cup L_2]$.
\end{claim}

\begin{proof}
Assume to the contrary that at most $\delta n$ vertices in $S_1$ have degree at least $k$ in $G_{n,k}'[S_1, L_1 \cup L_2]$. Then $e(S_1, L_1 \cup L_2) < (k-1)|S_1| + |L|\delta n \leq (k-1)n +\epsilon n$, because $|S_1| \leq n$ and by Lemma \ref{degrees of vertices in L}. This contradicts our assumption that $e(S_1, \{z\} \cup L_1\cup L_2) > (k + \epsilon)n$.
\end{proof}
Hence, there is a subset $D \subset S_1$ with at least $\delta n$ vertices such that every vertex in $D$ has degree at least $k$ in $G_{n, k}'[S_1, L_1\cup L_2]$. Since there are only at most $\binom{|L|}{k} \le \binom{500 k^2 / \alpha}{k}$ options for every vertex in $D$ to choose a set of $k$ neighbors from, we have that there exists some set of $k$ vertices in $L_1\cup L_2 \setminus \{z\}$ having a common neighborhood with at least $\delta n / \binom{|L|}{k} \ge \delta n / \binom{500 k^2 / \alpha}{k} = \frac{\epsilon \alpha}{500 k^2} n / \binom{500 k^2 / \alpha}{k} \ge 2k+2$ vertices. Thus, $K_{k, 2k+2} \subset G_{n,k}'[S_1, L_1 \cup L_2]$ and $K_{k+1, 2k+2} \subset G_{n,k}'[S_1, L_1 \cup L_2 \cup \{z\}]$, a contradiction by Lemma \ref{trees growing in bipartite graph}. Hence $e(S_1, \{z\} \cup L_1\cup L_2) \leq (k+\epsilon)n$.
\end{proof}

Next we show that all the vertices in $L'$ in fact have Perron weight close to the maximum. 

\begin{lemma}
\label{precise size of L}
For all vertices $v\in L'$, we have 
$d(v) \ge \left(1-\frac{1}{8k^3}\right)n$ and $\mathbf{\mathrm{x}}_v \geq 1- \frac{1}{16k^3}$. Moreover, $|L'| = k$.
\end{lemma}

\begin{proof}
Suppose we are able to show that $\mathbf{\mathrm{x}}_v \geq 1- \frac{1}{16k^3}$, then using Lemma \ref{mindegrees for vertices in L'} and by \eqref{choice of constants} we have that
$d(v) \ge  \left(1-\frac{1}{8k^3}\right)n$.
Then because every vertex $v \in L'$ has $d(v) \ge  \left(1-\frac{1}{8k^3}\right)n$, we must have that $|L'| \le k$ else $G_{n,k}'[S_1, L']$ contains a $K_{k+1, 2k+2}$, a contradiction by Lemma \ref{trees growing in bipartite graph}. 

Next, if $|L'| \le k-1$ then using Lemma \ref{neighborhood of x, L} and \eqref{upperbound on lambdasq 2} applied to the vertex $z$ we have 
\[kn \le \lambda^2 \le e(S_1, L' \setminus L) + e(S_1, L_1 \cup L_2)\eta + \frac{\epsilon n^2}{2} \le (k-1)n + (k+ \epsilon)n\eta + \frac{\epsilon^2 n}{2} < kn\]
where the last inequality holds by \eqref{choice of constants} and gives the contradiction. Thus $|L'| = k$. Thus, all we need to show is that $\mathbf{\mathrm{x}}_v \geq 1- \frac{1}{16k^3}$ for any $v \in L'$.

Now towards a contradiction assume that there is some vertex in $v \in L'$ such that $\mathbf{\mathrm{x}}_v < 1 - \frac{1}{16k^3}$.
Then refining \eqref{upperbound on lambdasq 2} with respect to the vertex $z$ we have that 
\begin{equation}
\begin{split}
 kn 
 & \le \lambda^2 < e(S_1, L_1(z) \cup L_2(z) \setminus \{v\}) + |N_1(z) \cap N_1(v)|\mathbf{\mathrm{x}}_v + \frac{\epsilon^2 n}{2}\\   
 & < (k+\epsilon)n - |S_1(z) \cap N_1(v)| + |N_1(z) \cap N_1(v)|\left(1 - \frac{1}{16k^3}\right) + \frac{\epsilon^2 n}{2}\\
 & = kn + \epsilon n + |L_1(z) \cap N_1(v)| - |N_1(z) \cap N_1(v)|\frac{1}{16k^3} + \frac{\epsilon^2 n}{2}.
\end{split}    
\end{equation}

Thus, using Lemma \ref{degrees of vertices in L} we have $\frac{|N_1(z) \cap N_1(v)|}{16k^3} < \epsilon n + \frac{\epsilon^2 n}{2} + |L| \le 2 \epsilon n$.

But, $v \in L'$, thus $\mathbf{\mathrm{x}}_v \ge \eta$ and $d(v) \ge \left(\eta - \epsilon\right)n$, and so $|N_1(v) \cap N_1(v)| \ge  \left(\eta - 2\epsilon\right)n > 32k^3 \epsilon n$ by \eqref{choice of constants}, a contradiction.
\end{proof}
Now we have $|L'| =  k$ and every vertex in $L'$ has degree at least $\left(1 - \frac{1}{8k^3}\right)n$. Thus, the common neighborhood of vertices in $L'$ has at least $\left(1 - \frac{1}{8k^2}\right)n$ vertices. Let $R$ denote the set of vertices in this common neighborhood.  Let $E$ be the set of remaining ``exceptional" vertices not in $L'$ or $R$. Thus $|E| \le \frac{n}{8k^2}$. We will now show that $E = \emptyset$ and thus $G_{n, k}'$ contains a large complete bipartite subgraph $K_{k, n-k}$.

\begin{lemma}
\label{minimum eigenweight in the neighborhood of a vertex}
For any vertex $v \in V(G_{n,k}')$, the Perron weight in the neighborhood of $v$ satisfies  $\sum_{w \sim v}\mathbf{\mathrm{x}}_w \geq k - \frac{1}{16k^2}$. 
\end{lemma}

\begin{proof}
Clearly if $v \in L'$, we have \[\sum_{w \sim v}\mathbf{\mathrm{x}}_w = \lambda \mathbf{\mathrm{x}}_v \ge \lambda \left(1 - \epsilon \right) \geq k - \frac{1}{16k^2}.\] If $v \in R$, then \[\sum_{w \sim v}\mathbf{\mathrm{x}}_w \ge \sum_{\substack{w \sim v\\ w \in L'}}\mathbf{\mathrm{x}}_w \ge k\left(1 - \frac{1}{16k^3}\right) = k - \frac{1}{16k^2}.\]

Finally, let $v \in E$. If $\sum_{w \sim u}\mathbf{\mathrm{x}}_w < k - \frac{1}{16k^2}$, consider the graph $H$ obtained from $G_{n,k}'$ by deleting all edges adjacent to $v$ and adding the edges $uv$ for all $u \in L'$. Now since $\sum_{w \sim v}\mathbf{\mathrm{x}}_w < k - \frac{1}{16k^2}$ we have that $x^TA(H)x > x^TA(G_{n,k}')x$, and so by the Rayleigh principle $\lambda(H) > \lambda(G_{n,k}')$. However, there are no new trees $T \in \mathcal{T}'_k$ that have isomorphic copies in $H$ but no isomorphic copies in $G_{n,k}'$. To see this, assume to the contrary that $T$ is a new tree having an isomorphic copy in $H$ but not in $G_{n, k}'$. Then $T$ has $2k+3$ vertices $v = v_1, v_2, \ldots v_{2k+3}$ and $v$ has at most $k$ neighbors in $T$ all of which lie in $L'$. Now since the common neighborhood of vertices in $L'$ has at least $\left(1 - \frac{1}{8k^2}\right)n > 2k+2$ vertices, $T$ must already have an isomorphic copy in $G_{n,k}'$, a contradiction.
\end{proof}

Let $K_{a, b}^+, K_{a, b}^{p},$ and $K_{a, b}^{m}$ denote the graphs obtained from the complete bipartite graph $K_{a, b}$ by adding into the part of size $b$ an edge $K_2$, a path $P_3$ on $3$ vertices, and a matching with two edges $K_2 \cup K_2$, respectively. 
\begin{lemma}
\label{trees growing in almost bipartite graphs}
The graph $K_{k, 2k+1}^+ := K_k \vee \left((2k-1)K_1 \cup K_2 \right)$ contains all trees in $\mathcal{T}_k$ and the graphs $K_{k, 2k+2}^{p} := K_k \vee \left((2k-1)K_1 \cup P_3 \right)$ and $K_{k, 2k+2}^{m}:= K_k \vee \left((2k-2)K_1 \cup 2 K_2 \right)$ contain all trees in $\mathcal{T}_k'$.
\end{lemma}

\begin{proof}
First we prove that $K_{k, 2k+1}^+$ contains all trees in $\mathcal{T}_k$. Say $T$ is a tree in $\mathcal{T}_k$. The tree $T$ is a bipartite graph which has smallest part of size less than or equal to $k+1$. If the smallest part has size less than $k+1$ then we are done, as $T \subset K_{k, 2k+1}$. Now if instead $T$ has both parts of size $k+1$ we let $v$ be a leaf of $T$. Then $T - \{v\}$ is a tree with $2k+1$ vertices with one part of size $k$ and the other of size $k+1$. Therefore, $T - \{v\} \subset K_{k, k+1}$ and $T \subset K_{k, k+2}^+ \subset K_{k, 2k+1}^+$.

Next, in case $T \in \mathcal{T}_k'$ then the tree $T$ is a bipartite graph whose smaller part has at most $k+1$ vertices. In case the smaller part of $T$ has at most $k$ vertices, then $T \subset K_{k, 2k+2}$ which is contained in both $K_{k, 2k+2}^{p}$ and $K_{k, 2k+2}^{m}$, and we are done. Now suppose $T$ has a smaller part of size $k+1$ and a larger part of size $k+2$. Then if there is a leaf in the smaller part, then by the same argument used for trees in $\mathcal{T}_k$ we can show that $T \subset  K_{k, k+3}^+ \subset K_{k, 2k+2}^+$ which is contained in both $K_{k, 2k+2}^{p}$ and $K_{k, 2k+2}^{m}$. If the smaller part has no leaves then every vertex in the smaller part has degree at least $2$ and since the sum of their degrees cannot be more than $2k+2$, we must have that each vertex in the smaller part of $T$ has degree $2$. Let $v$ be one such vertex of degree $2$ that lies in the smaller part. Then $T - \{v\} \subset K_{k, k+2}$ and $T \subset K_{k, k+3}^{p} \subset K_{k, 2k+2}^{p}$. 
Finally, let $l_1$ and $l_2$ be two leaves that lie in the larger part. Clearly they cannot have a common neighbor in the smaller part as $2k+3 > 3$ and $T$ is connected. Thus, $T - \{l_1, l_2\} \subset K_{k+1, k}$ and $T \subset K_{k, k+3}^{m} \subset K_{k, 2k+2}^{m}$ where we are now including $l_1$ and $l_2$ into the formerly smaller part. 
\end{proof}

It follows from Lemma \ref{trees growing in almost bipartite graphs} that $e(R) \le 1$. Moreover, any vertex in $E$ is adjacent to at most $2k+2$ vertices in $R$, else $K_{k+1, 2k+2} \subset G_{n, k}'$ a contradiction by Lemma \ref{trees growing in bipartite graph}. Finally, any vertex in $E$ is adjacent to at most $k-1$ vertices in $L'$ by the definition of $E$.

\begin{lemma}
\label{E empty set}
The set $E$ is empty and $G_{n,k}'$ contains the complete bipartite graph $K_{k, n-k}$.
\end{lemma}

\begin{proof}
Assume to the contrary that $E \neq \emptyset$. Recall that any vertex $r \in R$ satisfies $\mathbf{\mathrm{x}}_r < \eta$. Therefore, any vertex $v \in E$ must satisfy
\[\sum_{u \sim v} \mathbf{\mathrm{x}}_u = \lambda \mathbf{\mathrm{x}}_v = \sum_{\substack{u \sim v \\ u \in L' \cup R}}\mathbf{\mathrm{x}}_u + \sum_{\substack{u \sim v \\ u \in E}}\mathbf{\mathrm{x}}_u \le k-1 + \left(2k+2\right)\eta +\sum_{\substack{u \sim v \\ u \in E}}\mathbf{\mathrm{x}}_u.\]
Combining this with Lemma \ref{minimum eigenweight in the neighborhood of a vertex} gives 
\begin{equation}
\begin{split}
    \frac{\sum_{\substack{u \sim v \\ u \in E}} \mathbf{\mathrm{x}}_u}{\lambda \mathbf{\mathrm{x}}_v} \geq \frac{\lambda \mathbf{\mathrm{x}}_v - (k-1)-(2k+1)\eta}{\lambda \mathbf{\mathrm{x}}_v} \ge 1 - \dfrac{(k-1) + (2k+2) \eta}{k - \frac{1}{16k^2}} \ge \frac{4}{5k},
\end{split}
\end{equation}
where the last inequality follows from \eqref{choice of constants}. Now consider the matrix $B = A(G_{n,k}'[E])$ and vector $y := \mathbf{\mathrm{x}}_{|_E}$ (the restriction of the vector $\mathbf{\mathrm{x}}$ to the set $E$). We see that for any vertex $v \in E$

\[B \mathbf{\mathrm{y}}_v = \sum_{\substack{u \sim v \\ u \in E}}  \mathbf{\mathrm{x}}_u \ge \frac{4}{5k} \lambda \mathbf{\mathrm{x}}_v = \frac{4}{5k} \lambda \mathbf{\mathrm{y}}_v.\]

Hence, by Lemma \ref{lower bound for spectral radius of nonnegative matrices}, we have that $\lambda(B) \geq \frac{4}{5k}\lambda \ge \frac{4}{5}\sqrt{\frac{n}{k}}$. This is a contradiction to Lemma \ref{bounds on lambda} which gives $\lambda(B) \le \sqrt{5k|E|} \le \sqrt{5k\frac{n}{8k^2}} = \sqrt{\frac{5n}{8k}}$, else $E$ contains all trees in $\mathcal{T}_k'$.
\end{proof}

\section{Proof of Conjecture \ref{conjecture Nikiforov paths}}\label{section main proof}

It follows from Lemma \ref{E empty set} that both $G_{n,k}$ and $G_{n,k}'$ contain the complete bipartite graph $K_{k, n-k}$, where the part on $k$ vertices is the set $L'$ and the part on $n-k$ vertices is the set $R$. By Lemma \ref{trees growing in almost bipartite graphs} we have that $e(R) = 0$ in $G_{n,k}$ and $e(R) \leq 1$ in $G_{n,k}'$. Hence $G_{n,k} \subset S_{n,k}$ and $G_{n,k}' \subset S_{n,k}^+$. Since adding more edges will only increase the spectral radius, we see that the spectral radius is maximized if the vertices of $L'$ induce a clique $K_k$ and the number of edges in $R$ is as large as possible. Thus $G_{n,k} \cong S_{n,k}$ and $G_{n,k}' \cong S_{n,k}^+$
\qed

\section{Conclusion}
In this paper we proved Conjecture \ref{conjecture Nikiforov paths}, a spectral version of the Erd\H{o}s-S\'os conjecture. We would like however to highlight an important difference between the classical and spectral settings. If the Erd\H{o}s-S\'os conjecture is true, then when $t$ divides $n$, disjoint copies of $K_t$ show that the conjecture is best possible. Furthermore, this extremal construction does not contain any tree on $t+1$ vertices. On the other hand, $S_{n,k}$ and $S_{n,k}^+$ contain many trees on $2k+2$ and $2k+3$ vertices (and more). Therefore, it would be interesting to study $\mathrm{spex}(n, T)$ for various specific trees $T$, and we leave this as an open problem.
\bibliographystyle{plain}
	\bibliography{bib.bib}

\begin{thebibliography}{10}

\bibitem{BZ}
Abraham Berman and Xiao-Dong Zhang.
\newblock On the spectral radius of graphs with cut vertices.
\newblock {\em Journal of Combinatorial Theory, Series B}, 83(2):233--240,
  2001.

\bibitem{BLL}
B{\'e}la Bollob{\'a}s, Jonathan Lee, and Shoham Letzter.
\newblock Eigenvalues of subgraphs of the cube.
\newblock {\em European Journal of Combinatorics}, 70:125--148, 2018.

\bibitem{brandt}
Stephan Brandt and Edward Dobson.
\newblock The {E}rd{\H{o}}s--{S}{\'o}s conjecture for graphs of girth $5$.
\newblock {\em Discrete Mathematics}, 150(1-3):411--414, 1996.

\bibitem{BS}
Richard~A Brualdi and Ernie~S Solheid.
\newblock On the spectral radius of complementary acyclic matrices of zeros and
  ones.
\newblock {\em SIAM Journal on Algebraic Discrete Methods}, 7(2):265--272,
  1986.

\bibitem{evencycles}
Sebastian Cioab{\u{a}}, Dheer~Noal Desai, and Michael Tait.
\newblock The spectral even cycle problem.
\newblock {\em arXiv preprint arXiv:2205.00990}, 2022.

\bibitem{cioabua2022spectral}
Sebastian Cioab{\u{a}}, Dheer~Noal Desai, and Michael Tait.
\newblock The spectral radius of graphs with no odd wheels.
\newblock {\em European Journal of Combinatorics}, 99:103420, 2022.

\bibitem{FG}
Sebastian Cioab{\u{a}}, Lihua Feng, Michael Tait, and Xiao-Dong Zhang.
\newblock The maximum spectral radius of graphs without friendship subgraphs.
\newblock {\em The Electronic Journal of Combinatorics}, P4.22, 2020.

\bibitem{SpectralIntersectingCliques}
Dheer~Noal Desai, Liying Kang, Yongtao Li, Zhenyu Ni, Michael Tait, and Jing
  Wang.
\newblock Spectral extremal graphs for intersecting cliques.
\newblock {\em Linear Algebra and its Applications}, 2022.

\bibitem{EZ}
Mark~N Ellingham and Xiaoya Zha.
\newblock The spectral radius of graphs on surfaces.
\newblock {\em Journal of Combinatorial Theory, Series B}, 78(1):45--56, 2000.

\bibitem{erdos1965}
Paul Erd{\H{o}}s.
\newblock Extremal problems in graph theory, theory of graphs and its
  applications (m. fiedler, ed.), 1965.

\bibitem{ES66}
Paul Erd{\H{o}}s and Mikl{\'o}s Simonovits.
\newblock A limit theorem in graph theory.
\newblock In {\em Studia Sci. Math. Hung}, 1965.

\bibitem{ES46}
Paul Erd{\H{o}}s and Arthur~H Stone.
\newblock On the structure of linear graphs.
\newblock {\em Bulletin of the American Mathematical Society},
  52(12):1087--1091, 1946.

\bibitem{FN}
Miroslav Fiedler and Vladimir Nikiforov.
\newblock Spectral radius and hamiltonicity of graphs.
\newblock {\em Linear Algebra and its Applications}, 432(9):2170--2173, 2010.

\bibitem{degeneratesurvey}
Zolt{\'a}n F{\"u}redi and Mikl{\'o}s Simonovits.
\newblock The history of degenerate (bipartite) extremal graph problems.
\newblock In {\em Erd{\H{o}}s Centennial}, pages 169--264. Springer, 2013.

\bibitem{GZ}
Agnieszka Goerlich and Andrzej {\.Z}ak.
\newblock On {E}rd{\H{o}}s-{S}{\'o}s conjecture for trees of large size.
\newblock {\em The Electronic Journal of Combinatorics}, pages P1--52, 2016.

\bibitem{hou}
Xinmin Hou, Boyuan Liu, Shicheng Wang, Jun Gao, and Chenhui Lv.
\newblock The spectral radius of graphs without trees of diameter at most four.
\newblock {\em Linear and Multilinear Algebra}, 69(8):1407--1414, 2021.

\bibitem{Yongtao21}
Yongtao Li and Yuejian Peng.
\newblock The spectral radius of graphs with no intersecting odd cycles.
\newblock {\em arXiv preprint arXiv:2106.00587}, 2021.

\bibitem{LBW2}
Xiangxiang Liu, Hajo Broersma, and Ligong Wang.
\newblock On a conjecture of {N}ikiforov involving a spectral radius condition
  for a graph to contain all trees.
\newblock {\em arXiv preprint arXiv:2112.13253}, 2021.

\bibitem{LBW}
Xiangxiang Liu, Hajo Broersma, and Ligong Wang.
\newblock Spectral radius conditions for the existence of all subtrees of
  diameter at most four.
\newblock {\em arXiv preprint arXiv:2109.11546}, 2021.

\bibitem{LLT}
Mei Lu, Huiqing Liu, and Feng Tian.
\newblock A new upper bound for the spectral radius of graphs with girth at
  least {$5$}.
\newblock {\em Linear Algebra and its Applications}, 414(2-3):512--516, 2006.

\bibitem{M}
Andrew McLennan.
\newblock The {E}rd{\H{o}}s-{S}{\'o}s conjecture for trees of diameter four.
\newblock {\em Journal of Graph Theory}, 49(4):291--301, 2005.

\bibitem{Nikiforov07}
Vladimir Nikiforov.
\newblock Bounds on graph eigenvalues {II}.
\newblock {\em Linear {A}lgebra and its {A}pplications}, 427(2):183--189, 2007.

\bibitem{NO}
Vladimir Nikiforov.
\newblock A spectral condition for odd cycles in graphs.
\newblock {\em Linear Algebra and its Applications}, 428(7):1492--1498, 2008.

\bibitem{Nikiforovpaths}
Vladimir Nikiforov.
\newblock The spectral radius of graphs without paths and cycles of specified
  length.
\newblock {\em Linear Algebra and its Applications}, 432(9):2243--2256, 2010.

\bibitem{nosal1970eigenvalues}
Eva Nosal.
\newblock Eigenvalues of graphs.
\newblock Master's thesis, University of Calgary, 1970.

\bibitem{SW}
Jean-Fran{\c{c}}ois Sacl{\'e} and Mariusz Wo{\'z}niak.
\newblock The {E}rd{\H{o}}s--{S}{\'o}s conjecture for graphs without {${C}_4$}.
\newblock {\em Journal of Combinatorial Theory, Series B}, 70(2):367--372,
  1997.

\bibitem{SSbook}
Lingsheng Shi and Zhipeng Song.
\newblock Upper bounds on the spectral radius of book-free and/or
  {$K_{2,l}$}-free graphs.
\newblock {\em Linear Algebra and its Applications}, 420(2-3):526--529, 2007.

\bibitem{STY}
PJ~Slater, SK~Teo, and HP~Yap.
\newblock Packing a tree with a graph of the same size.
\newblock {\em Journal of graph theory}, 9(2):213--216, 1985.

\bibitem{S}
Richard~P Stanley.
\newblock A bound on the spectral radius of graphs with e edges.
\newblock {\em Linear Algebra and its Applications}, 87:267--269, 1987.

\bibitem{SAH}
Dragan Stevanovi{\'c}, Mustapha Aouchiche, and Pierre Hansen.
\newblock On the spectral radius of graphs with a given domination number.
\newblock {\em Linear Algebra and its Applications}, 428(8-9):1854--1864, 2008.

\bibitem{Tiner}
Gary Tiner.
\newblock On the {E}rd{\H{o}}s-s{\'o}s conjecture for graphs on $n= k+ 3$
  vertices.
\newblock {\em Ars Comb}, 95:143--150, 2010.

\bibitem{Wilf86}
Herbert~S Wilf.
\newblock Spectral bounds for the clique and independence numbers of graphs.
\newblock {\em Journal of Combinatorial Theory, Series B}, 40(1):113--117,
  1986.

\bibitem{wozniak}
Mariusz Wo{\'z}niak.
\newblock On the {E}rd{\H{o}}s-{S}{\'o}s conjecture.
\newblock {\em Journal of Graph Theory}, 21(2):229--234, 1996.

\bibitem{YWZ}
Wanlian Yuan, Bing Wang, and Mingqing Zhai.
\newblock On the spectral radii of graphs without given cycles.
\newblock {\em The Electronic Journal of Linear Algebra}, 23:599--606, 2012.

\bibitem{ZW}
Mingqing Zhai and Bing Wang.
\newblock Proof of a conjecture on the spectral radius of ${C_4}$-free graphs.
\newblock {\em Linear Algebra and its Applications}, 437(7):1641--1647, 2012.

\bibitem{ZWF}
Mingqing Zhai, Bing Wang, and Longfei Fang.
\newblock The spectral {T}ur{\'a}n problem about graphs with no 6-cycle.
\newblock {\em Linear Algebra and its Applications}, 590:22--31, 2020.

\end{thebibliography}
\end{document}